\documentclass[12pt]{amsart}
\usepackage{amscd,amsmath,amsthm,amssymb}
\usepackage{amsfonts,amssymb,amscd,amsmath,enumerate,verbatim}
\usepackage[left]{lineno}
\usepackage{pstricks, pst-plot,pst-3d}
\usepackage{tikz}
\usepackage{epsfig}
\newpsstyle{fatline}{linewidth=1.5pt}
\definecolor{verylight}{gray}{0.97}
\definecolor{light}{gray}{0.9}
\definecolor{medium}{gray}{0.85}
\definecolor{dark}{gray}{0.6}
\def\NZQ{\mathbb}               

\def\ZZ{{\NZQ Z}}

\def\Ac{{\mathcal A}}

\def\opn#1#2{\def#1{\operatorname{#2}}} 
\opn\chara{char} \opn\length{\ell} \opn\pd{pd} \opn\rk{rk}
\opn\projdim{proj\,dim} \opn\injdim{inj\,dim} \opn\rank{rank}
\opn\depth{depth} \opn\grade{grade} \opn\height{height}
\opn\reg{reg}
\opn\embdim{emb\,dim} \opn\codim{codim}
\opn\Cl{Cl}

\opn\Tr{Tr} \opn\bigrank{big\,rank}
\opn\superheight{superheight}\opn\lcm{lcm}
\opn\trdeg{tr\,deg}
	\opn\reg{reg} \opn\lreg{lreg} \opn\ini{in} \opn\lpd{lpd}
	\opn\size{size} \opn\sdepth{sdepth}
	\opn\link{link}
    \opn\fdepth{fdepth}\opn\lex{lex}
	\opn\tr{tr}\opn\del{del}
	\opn\type{type}
	\opn\gap{gap}
	\opn\arithdeg{arith-deg}
	\opn\revlex{revlex}
	\opn\div{div} \opn\Div{Div} \opn\cl{cl} \opn\Cl{Cl}
	\opn\Spec{Spec} \opn\Supp{Supp} \opn\supp{supp} \opn\Sing{Sing}
	\opn\Ass{Ass} \opn\Min{Min}\opn\Mon{Mon}
	\opn\Ann{Ann} \opn\Rad{Rad} \opn\Soc{Soc}
	\opn\Im{Im} \opn\Ker{Ker} \opn\Coker{Coker} \opn\Am{Am}
	\opn\Hom{Hom} \opn\Tor{Tor} \opn\Ext{Ext} \opn\End{End}
	\opn\Aut{Aut} \opn\id{id}
	
	\opn\nat{nat}
	\opn\pff{pf}
	\opn\Pf{Pf} \opn\GL{GL} \opn\SL{SL} \opn\mod{mod} \opn\ord{ord}
	\opn\Gin{Gin} \opn\Hilb{Hilb}\opn\sort{sort}
	\opn\PF{PF}\opn\Ap{Ap}
	\opn\mult{mult}
	\opn\bight{bight}
    \opn\adj{adj}
	\opn\div{div}
	\opn\Div{Div}
	\opn\aff{aff}
	\opn\relint{relint} \opn\st{st}
	\opn\lk{lk} \opn\cn{cn} \opn\core{core} \opn\vol{vol}  \opn\inp{inp} \opn\nilpot{nilpot}
	\opn\link{link} \opn\star{star}\opn\lex{lex}\opn\set{set}
	\opn\width{wd}
	\opn\Fr{F}
	\opn\QF{QF}
	\opn\G{G}
	\opn\type{type}\opn\res{res}
	\opn\conv{conv}
	\opn\Deg{Deg}
	\opn\Sym{Sym}
	\opn\Con{Con}
	\opn\gr{gr}
	
	\def\pot#1#2{#1[\kern-0.28ex[#2]\kern-0.28ex]}

	\opn\dirlim{\underrightarrow{\lim}}
	\opn\inivlim{\underleftarrow{\lim}}
	\def\Implies{\ifmmode\Longrightarrow \else
		\unskip${}\Longrightarrow{}$\ignorespaces\fi}
	\def\implies{\ifmmode\Rightarrow \else
		\unskip${}\Rightarrow{}$\ignorespaces\fi}
	\def\iff{\ifmmode\Longleftrightarrow \else
		\unskip${}\Longleftrightarrow{}$\ignorespaces\fi}

	\let\:=\colon
	\newtheorem{Theorem}{Theorem}[section]
	\newtheorem{Lemma}[Theorem]{Lemma}
	\newtheorem{Corollary}[Theorem]{Corollary}

	\newtheorem{Example}[Theorem]{Example}

	\let\epsilon\varepsilon
	\let\kappa=\varkappa
	\def\qed{\ifhmode\textqed\fi
		\ifmmode\ifinner\quad\qedsymbol\else\dispqed\fi\fi}
	\def\textqed{\unskip\nobreak\penalty50
		\hskip2em\hbox{}\nobreak\hfil\qedsymbol
		\parfillskip=0pt \finalhyphendemerits=0}
	\def\dispqed{\rlap{\qquad\qedsymbol}}
	
	\opn\dis{dis}
	\def\pnt{{\raise0.5mm\hbox{\large\bf.}}}
	
	\opn\Lex{Lex}

\begin{document}
\title[Componentwise linear monomial ideals]{Componentwise linear monomial ideals}

\author[T.~Hibi]{Takayuki Hibi}
\author[S.~A.~ Seyed Fakhari]{Seyed Amin Seyed Fakhari}

\address{(Takayuki Hibi) Department of Pure and Applied Mathematics, Graduate School of Information Science and Technology, The University of Osaka, Suita, Osaka 565--0871, Japan}
\email{hibi@math.sci.osaka-u.ac.jp}
\address{(Seyed Amin Seyed Fakhari) Departamento de Matem\'aticas, Universidad de los Andes, Bogot\'a, Colombia}
\email{s.seyedfakhari@uniandes.edu.co}



\begin{abstract}
Let $S=K[x_1,\ldots,x_n]$ denote the polynomial ring in $n$ variables over a field $K$ with $\deg x_1=\cdots =\deg x_n = 1$ and ${\bf a} = (a_1,\ldots,a_n) \in \ZZ_{>0}^n$.  Given a squarefree monomial $u=x_{i_1} \cdots x_{i_d}$ of $S$ with $1 \leq i_1 < \cdots < i_d \leq n$, we set $u^{[{\bf a}]}:=x_{i_1}^{a_{i_1}}\cdots x_{i_d}^{a_{i_d}}$.  Let $I$ be a squarefree monomial ideal of $S$ and $G(I)$ its unique minimal set of monomial generators.  We introduce the monomial ideal $I^{[{\bf a}]}$ with $G(I^{[{\bf a}]})=\{u^{[{\bf a}]} : u \in G(I)\}$.  In the present paper, componentwise linearity of a squarefree monomial ideal $I$ and that of $I^{[{\bf a}]}$ is studied.  
\end{abstract}	
\maketitle
\thispagestyle{empty}

\section*{Introduction}
Let $S=K[x_1,\ldots,x_n]$ denote the polynomial ring in $n$ variables over a field $K$ with $\deg x_1=\cdots =\deg x_n = 1$  and ${\bf a} = (a_1,\ldots,a_n) \in \ZZ_{>0}^n$.  Given a squarefree monomial $u=x_{i_1} \cdots x_{i_d}$ of $S$ with $1 \leq i_1 < \cdots < i_d \leq n$, we set $u^{[{\bf a}]}:=x_{i_1}^{a_{i_1}}\cdots x_{i_d}^{a_{i_d}}$. Let $I$ be a squarefree monomial ideal of $S$ and $G(I)$ its unique minimal set of monomial generators.  We introduce the monomial ideal $I^{[{\bf a}]}$ with $G(I^{[{\bf a}]})=\{u^{[{\bf a}]} : u \in G(I)\}$. 

Let $I$ be a homogeneous ideal of $S$ and $I_{\langle j \rangle}$ the ideal of $S$ generated by all homogeneous polynomials of degree $j$ belonging to $I$.  We say that $I$ is {\em componentwise linear} \cite{componentwise} if $I_{\langle j \rangle}$ has a linear resolution for all $j$. 

In the present paper, componentwise linearity of a squarefree monomial ideal $I$ and that of $I^{[{\bf a}]}$ is studied.  First, in section \ref{sec1}, it is shown that if $I$ is a squarefree monomial ideal and if $I^{[{\bf a}]}$ is componentwise linear, then $I$ is componentwise linear (Corollary \ref{Cor1.3}).  Second, in Section \ref{sec2}, given a squarefree lexsegment ideal $I$ of $S$ and ${\bf a} = (a_1,\ldots,a_n) \in \ZZ_{>0}^n$, a classification for $I^{[{\bf a}]}$ to be componentwise linear is obtained (Theorem \ref{lexsegment}).  As a corollary, it is shown that for a squarefree Veronese ideal $I$ of $S$ of degree $d$, $I^{[{\bf a}]}$ is componentwise linear if and only if at most one component of ${\bf a} = (a_1,\ldots,a_n) \in \ZZ_{>0}^n$ is bigger than one (Corollary \ref{veronese}).  On the other hand, in Section \ref{sec3}, given a finite graph $G$ on $n$ vertices and ${\bf a} = (a_1,\ldots,a_n) \in \ZZ_{>0}^n$, a characterization for $I(G)^{[{\bf a}]}$, where $I(G)$ is the edge ideal of $G$, to be componentwise linear is obtained (Theorem \ref{chordal}).

\section{Compnentwise linearity of $I^{[{\bf a}]}$ and $I$} \label{sec1}

Let $I$ be a squarefree monomial ideal of $S$ and ${\bf a} = (a_1,\ldots,a_n) \in \ZZ_{>0}^n$.  Suppose that $I^{[{\bf a}]}$ is componentwise linear.  One can ask if $I$ componentwise linear. The goal of this section is to answer this question (see Corollary \ref{Cor1.3}).

Recall that the least common multiple of two monomials $u,v$ is denoted by $\lcm(u,v)$. Let $I$ be a monomial ideal of $S$ generated in a single degree $d$. For any pair of monomials $u,v\in G(I)$ the {\em relation graph} $G_I^{(u,v)}$ is defined as follows. The vertices of $G_I^{(u,v)}$ are those monomials $w\in G(I)$ which divide $\lcm(u,v)$. Two vertices $w,w'$ ae adjacent in $G_I^{(u,v)}$ if $\deg(\lcm(w,w'))=d+1$. The following result will be used frequently in this paper.

\begin{Lemma}\cite[Proposition 2.1]{BHZ} \label{linrel}
Let $I$ be a monomial ideal of $S$ which is generated in a single degree $d$. Then $I$ is linearly related if and only if for all $u,v\in G(I)$, the graph $G_I^{(u,v)}$ is connected.
\end{Lemma}

Let $I, I_1, I_2$ be monomial ideals of $S$ such that $G(I)$ is the disjoint union of $G(I_1)$ and $G(I_2)$. We say that $I=I_1+I_2$ is a {\em Betti splitting} \cite{FHVT} if for any pair of integers $i,j$, one has$$\beta_{i.j}(I)=\beta_{i,j}(I_1)+\beta_{i,j}(I_2)+\beta_{i-1,j}(I_1\cap I_2).$$

We now prove the first main result of this paper.

\begin{Theorem}
\label{main}
 Let $I$ be a squarefree monomial ideal of S and ${\bf a} = (a_1,\ldots,a_n) \in \ZZ_{>0}^n$.  Suppose that there is $1 \leq k \leq n$ with $a_k \geq 2$.  Let ${\bf a'}\in \ZZ_{>0}^n$ be the vector obtained from ${\bf a}$ by replacing $a_k$ with $a_k-1$. If $I^{[{\bf a}]}$ is componentwise linear, then $I^{[{\bf a'}]}$ is componentwise linear.
\end{Theorem}

\begin{proof}
Dividing $I$ by the greatest common divisor of its generators, we may assume that $\height{I}\geq 2$. Set $L:=I^{[{\bf a}]}$ and $J:=I^{[{\bf a'}]}$.  Let, say, $k=1$. Note that $x_1J\subseteq L$. We must prove that for each integer $j\geq 1$, the ideal $J_{\langle j \rangle}$ has a linear resolution. Let $L_1$ be the ideal of $S$ generated by those monomials in $G(L_{\langle j+1 \rangle})$ which are divisible by $x_1$. Let $L_2$ denote the ideal generated by $G(L_{\langle j+1 \rangle})\setminus G(L_1)$. One has  $L_{\langle j+1 \rangle}=L_1+L_2$.

\medskip

\noindent
{\bf (Claim 1.)} $L_1=x_1J_{\langle j \rangle}$.

\medskip

\noindent
{\it Proof of Claim 1.} The inclusion $x_1J_{\langle j \rangle}\subseteq L_1$ follows from $x_1J\subseteq L$.  To prove 
$L_1 \subseteq x_1J_{\langle j \rangle}$,
let $w\in G(L_1)$. In particular, $w$ is divisible by $x_1$ and ${\rm deg}(w)=j+1$.  It is enough to show that $w/x_1\in J$. Since $w\in L=I^{[{\bf a}]}$, there is $w_0\in G(I)$ for which $w_0^{[{\bf a}]}$ divides $w$. If $x_1$ does not divide $w_0$, then $w_0^{[{\bf a}]}=w_0^{[{\bf a'}]}$ and so, $w/x_1$ is divisible by $w_0^{[{\bf a'}]}$.  Hence $w/x_1\in J$. If $x_1$ divides $w_0$, then $w_0^{[{\bf a}]}=x_1w_0^{[{\bf a'}]}$. Hence,  $w/x_1$ is divisible by $w_0^{[{\bf a'}]}=w_0^{[{\bf a}]}/x_1$. Again, $w/x_1\in J$. This proves Claim 1.

\medskip

\noindent
{\bf (Claim 2.)} $L_1\cap L_2=\big(x_1u : u\in G(L_2)\setminus G(L)\big)$.

\medskip

\noindent
{\it Proof of Claim 2.} First, we show that the right-hand side is contained in the left-hand side. Clearly, for each $u\in G(L_2)\setminus G(L)$, one has $x_1u\in L_2$. So, we only need to prove that $x_1u\in L_1$. Since $u\in L\setminus G(L)$, there is a variable $x_t$ dividing $u$ for which $u/x_t\in L$.  One has $x_1u/x_t\in L_1$ which implies that $x_1u\in L_1$.

Second, we show that the left-hand side is contained in the right-hand side. Let $v\in G(L_1\cap L_2)$. Then there are $v_1\in G(L_1)$ and $v_2\in G(L_2)$ for which $v=\lcm{(v_1,v_2)}$.  It follows from the definition of $L_1$ that $x_1$ divides $v_1$. Set $t:=\deg_{x_1}(v_1)\geq 1$. As $\deg(v_1)=\deg(v_2)=j+1$ and $x_1$ does not divide $v_2$, there is a variable $x_r$ for which $\deg_{x_r}(v_2)>\deg_{x_r}(v_1)$. If $t=1$, it follows from $a_1\geq 2$ and the definition of $L$ that $v_1/x_1\in L$. Therefore, $v_1x_r/x_1\in G(L_2)\setminus G(L)$. One has $v_1x_r=\lcm(v_1, v_1x_r/x_1)\in L_1\cap L_2$. Since $v\in G(L_1\cap L_2)$ is divisible by $v_1x_r$, we deduce that$$v=v_1x_r=x_1(v_1x_r/x_1)\in \big(x_1u : u\in G(L_2)\setminus G(L)\big).$$

Suppose that $t\geq 2$. Since $L_{\langle j+1 \rangle}$ has a linear resolution, it is linearly related.  So, we conclude from Lemma \ref{linrel} that the relation graph $G_{L_{\langle j+1 \rangle}}^{(v_1,v_2)}$ is connected.  Consequently, there is $v_1'\in G(L_{<j+1>})$ dividing $\lcm{(v_1,v_2)}=v$ for which $\deg_{x_1}(v_1')=t-1\geq 1$. Thus, $v_1'\in G(L_1)$ and $\lcm(v_1',v_2)\in L_1\cap L_2$ strictly divides $v$, as$$\deg_{x_1}(\lcm(v_1',v_2))=t-1<\deg_{x_1}(v_1)=\deg_{x_1}(v).$$This contradicts $v\in G(L_1\cap L_2)$, and completes the proof of Claim 2.

\medskip

We show that $L_{\langle j+1 \rangle}=L_1+L_2$ is a Betti splitting. Following \cite[Proposition 2.1]{FHVT}, our mission is to prove that for each $p,q$, the natural maps$$\phi_1: \Tor_p^S{(K, L_1\cap L_2)}_q\longrightarrow\Tor_p^S{(K, L_1)}_q$$and$$\phi_2: \Tor_p^S{(K, L_1\cap L_2)}_q\longrightarrow\Tor_p^S{(K, L_2)}_q$$ are zero.

Clearly, $\phi_2$ vanishes.  In fact, 
for ${\bf b}=(b_1, \ldots, b_n)\in \ZZ^n$ with $\Tor_p^S{(K, L_2)}_{{\bf b}}\neq 0$, one has $b_1=0$.  Furthermore, Claim 2 guarantees that each monomial generator of $L_1\cap L_2$ is divisible by $x_1$. This shows that $\phi_2$ is the zero map.

We prove that $\phi_1$ is zero.  Let $a\in \Tor_p^S{(K, L_1\cap L_2)}_q$. We show that $\phi_1(a)=0$. Suppose on the contrary, that $\phi_1(a)\neq 0$. Let ${\bf c} =(c_1, \ldots, c_n)$ denote the multidegree of $a$.  Again, Claim 2 says that each monomial generator of $L_1\cap L_2$ is divisible by $x_1$, but not by $x_1^2$. This yields that $c_1=1$. It then follows that 
\begin{eqnarray}
\label{ggggg}   
\Tor_p^S{(K, L_1)}_{\bf c}\cong \Tor_p^S{(K, L_1')}_{\bf c},
\end{eqnarray}
where $L_1'=(\{u\in L_1 :\deg_{x_1}(u)=1\})$. Since $0\neq \phi_1(a)\in \Tor_p^S{(K, L_1)}_{\bf c}$, the above isomorphism (\ref{ggggg}) yields that $\Tor_p^S{(K, L_1')}_{\bf c}\neq 0$. On the other hand, $L_1'=x_1L_1''$, for some monomial ideal $L_1''$ of $S$. It follows from $a_1\geq 2$ and the definition of $L$ that $L_1''$ is the ideal generated by those monomial generators of $L_{\langle j \rangle}$ which are not divisible by $x_1$.  Since $L_{\langle j \rangle}$ has a linear resolution, \cite[Lemma 4.4]{HHZ} says that $L_1''$ has a linear resolution. As a consequence, $L_1'$ has a linear resolution, which implies that $|{\bf c}|=p+j+1$. However, this is impossible, as ${\bf c}=\deg(a)$, $a\in \Tor_p^S{(K, L_1\cap L_2)}_q$ and, by Claim 2, $L_1\cap L_2$ is generated in degree $j+2$.

Therefore, $L_{\langle j+1 \rangle}=L_1+L_2$ is a Betti splitting.  It then follows that for each $p,q$,
\begin{eqnarray}
\label{hhhhh}
\beta_{p,q}(L_{\langle j+1 \rangle})=\beta_{p,q}(L_1)+\beta_{p,q}(L_2)+\beta_{p-1,q}(L_1\cap L_2).
\end{eqnarray}
Since $L_{\langle j+1 \rangle}$ has a linear resolution, it follows from the above equality (\ref{hhhhh}) that $L_1$ has linear resolution. Finally, combining with Claim 1, we conclude that $J_{\langle j \rangle}$ has a linear resolution, as desired.
\, \, \, \, \, \, \, 
\end{proof}

\begin{Corollary}
\label{Cor1.3}
 Let $I$ be a squarefree monomial ideal of S and ${\bf a} \in \ZZ_{>0}^n$.  Suppose that $I^{[{\bf a}]}$ is componentwise linear.  Then $I$ is componentwise linear.
\end{Corollary}

\section{Squarefree lexsegment ideals} \label{sec2}

Let $S=K[x_1,\ldots,x_n]$ denote the polynomial ring in $n$ variables over a field $K$ with $\deg x_1 = \cdots =\deg x_n= 1$ and $<_{\rm lex}$ the lexicographic order (\cite[Example 2.1.2 (a)]{HHgtm260}) on $S$ induced by the ordering $x_1 > \cdots > x_n$.
Recall from \cite{lexsegment} that a squarefree monomial ideal $I$ of $S$ (not necessarily generated in a single  degree) is {\em squarefree lexsegment} if for any squarefree monomial $u\in I$ and any squarefree monomial $v\in S$ with $\deg u = \deg v$ and $u <_{\rm lex} v$, one has $v \in I$. It is well known that any squarefree lexsegment ideal is componentwise linear.

\begin{Lemma} \label{1necessary}
Let $I$ be a squarefree lexsegment ideal of $S$ and suppose that an integer $\ell\in [n]$ enjoys the property that there exist $u,v\in G(I)$ for which 
(i) $x_{\ell}$ divides $v$, 
(ii) $v/x_{\ell}$ divides $u$ and 
(iii) $\deg u> \deg v$. 
Then, given ${\bf a} = (a_1,\ldots,a_n) \in \ZZ_{>0}^n$ with $a_{\ell}\geq 2$, the ideal $I^{[{\bf a}]}$ is not componentwise linear.
\end{Lemma}

\begin{proof}
Set $J:=I^{[{\bf a}]}$, $d:=\deg u^{[{\bf a}]}$ and $d':=\deg v^{[{\bf a}]}$. Since $u$ and $v$ are squarefree monomials belonging to $G(I)$, it follows from  (ii) that $x_{\ell}$ does not divide $u$. Moreover, we conclude from (iii) that there are distinct variables $x_p, x_q$ which divide $u$ but not $v$. By contradiction, assume that $J$ is componentwise linear. The following two cases arise.

\medskip

\noindent
{\bf (Case 1.)} Let $d\geq d'$. Set $v':=x_{\ell}^{d-d'}v^{[{\bf a}]}$, $k:=d=\deg v'$. Since $J$ is componentwise linear, $J_{\langle k\rangle}$ has a linear resolution. In particular, $J_{\langle k\rangle}$ is linearly related. We deduce from Lemma \ref{linrel} that the relation graph $G_{J_{\langle k\rangle}}^{(u^{[{\bf a}]},v')}$ is connected. Hence, there is a monomial $w\in J_{\langle k\rangle}$ in the variables dividing $u$ or $v$ with $\deg_{x_\ell}(w)=a_{\ell}-1$. It follows from the definition of $J=I^{[{\bf a}]}$ that $w/x_{\ell}^{a_{\ell}-1}\in J$. It then follows from (ii) that $w/x_{\ell}^{a_{\ell}-1}$ is divisible by $u^{[{\bf a}]}$.  This yields that $\deg(w)\geq (a_{\ell}-1)+d>k$, since $a_{\ell}\geq 2$.  This is a contradiction.

\medskip

\noindent
{\bf (Case 2.)} Let $d< d'$. It follows from (ii) that$$d'-d\leq a_{\ell}-a_p-a_q\leq a_{\ell}-2.$$Set $u':=x_p^{d'-d}u^{[{\bf a}]}$, $k:=d'=\deg u'$. By the same argument as in Case 1, the relation graph $G_{J_{\langle k\rangle}}^{(u',v^{[{\bf a}]})}$ is connected. Hence, there is a monomial $w\in J_{\langle k\rangle}$ in the variables dividing $u$ or $v$ with $\deg_{x_\ell}(w)=a_{\ell}-1$. It follows from the definition of $J=I^{[{\bf a}]}$ that $w/x_{\ell}^{a_{\ell}-1}\in J$. It then follows from (ii) that $w/x_{\ell}^{a_{\ell}-1}$ is divisible by $u^{[{\bf a}]}$. This yields that $$\deg(w)\geq (a_{\ell}-1)+d>d'=k,$$ since $d-d'\leq a_{\ell}-2$.  This is a contradiction.
\hspace{5.5cm}
\end{proof}


\begin{Lemma} \label{sufficient}
Let $I$ be a squarefree lexsegment ideal and ${\bf a} = (a_1,\ldots,a_n) \in \ZZ_{>0}^n$.  Suppose the conditions $(\sharp)$ and $(*)$ as follows:

\smallskip

$(\sharp)$ For each $\ell\in [n]$ with $a_{\ell}\geq 2$, there do not exist $u,v\in G(I)$ satisfying (i), (ii) and (iii) of Lemma \ref{1necessary}. 

\smallskip

$(*)$ There exists a (possibly empty) subset $\{i_1, \ldots, i_s\}\subseteq [n]$ with $i_1<\cdots < i_s$ satisfying the conditions as follows:

\begin{itemize}
\item
$a_{i_1} \geq 2, \ldots, a_{i_s}\geq 2$; 
\item
$a_k=1$ for each $k\in [n]\setminus \{i_1,\ldots, i_s\}$; 
\item
If $u\in G(I)$ is divisible by $x_{i_j}$ ($2\leq j\leq s$), then it is divisible by $x_{i_{j-1}}$ as well. 
\end{itemize}
Then $I^{[{\bf a}]}$ is a componentwise linear ideal.
\end{Lemma}

\begin{proof}
Set $J:=I^{[{\bf a}]}$. Using induction on $s$, we show that $J$ is a weakly polymatroidal ideal \cite[Pages 266--267]{HHgtm260} with respect to the ordering of the variables as follows:$$x_1>x_2\cdots > x_{i_1-1}> x_{i_1+1}>\cdots > x_{i_2-1}>x_{i_2+1}> \cdots >x_n>x_{i_1}>\cdots > x_{i_s}.$$ The case $s=0$ is obvious, as $I$ is a weakly polymatroidal ideal. So, assume that $s\geq 1$. Let $G(I) = \{u_1,\ldots,u_t\} \sqcup \{v_1,\ldots,v_m\}$, where each $u_k$ is not divisible by $x_{i_s}$ and where each $v_\ell$ is divisible by $x_{i_s}$ (and so, by $x_{i_1}, \ldots, x_{i_{s-1}}$). 
Then $(u_1, \ldots, u_t)$ is a squarefree lexsegment ideal of $K[x_1, \ldots, x_{i_s-1},x_{i_s+1}, \ldots, x_n]$. It follows from the induction hypothesis  that  the ideal $(u_1^{[{\bf a}]}, \ldots, u_t^{[{\bf a}]})$ is a weakly polymatroidal ideal.

In order to show that $J$ is a weakly polymatroidal ideal, we choose monomials $w_1^{[{\bf a}]}, w_2^{[{\bf a}]}\in G(J)$. Since $(u_1^{[{\bf a}]}, \ldots, u_t^{[{\bf a}]})$ is a weakly polymatroidal ideal, without loss of generality, we may assume that $w_2^{[{\bf a}]}\in \{v_1^{[{\bf a}]}, \ldots, v_m^{[{\bf a}]}\}$. Since $I$ is a squarefree monomial ideal and $w_2$ is divisible by $x_{i_1}, \ldots, x_{i_s}$, it is impossible that $\deg_{x_k}(w_1^{[{\bf a}]})=\deg_{x_k}(w_2^{[{\bf a}]})$ for each $k\in [n]\setminus \{i_1, \ldots, i_s\}$. Let $t$ be the smallest integer belonging to $[n]\setminus\{i_1, \ldots, i_s\}$ for which $\deg_{x_t}(w_1)\neq \deg_{x_t}(w_2)$. The following cases arise.

\medskip

\noindent
{\bf (Case 1.)} Let $w_1^{[{\bf a}]}\in \{v_1^{[{\bf a}]}, \ldots, v_m^{[{\bf a}]}\}$. By symmetry, assume that $\deg_{x_t}(w_1) < \deg_{x_t}(w_2)$. Then there is a variable $x_r$ dividing $w_1$ but not $w_2$. In particular, $x_r\notin \{x_{i_1}, \ldots, x_{i_s}\}$. It follows from the choice of $t$ that $r>t$. Since $w_1<_{\rm lex} x_tw_1/x_r$ and $I$ is a squarefree lexsegment ideal, one has $x_tw_1/x_r\in I$. This yields that $x_tw_1^{[{\bf a}]}/x_r\in J$, as desired.

\medskip

\noindent
{\bf (Case 2.)} Let $w_1^{[{\bf a}]}\in \{u_1^{[{\bf a}]}, \ldots, u_t^{[{\bf a}]}\}$. First, suppose that $\deg_{x_t}(w_1) > \deg_{x_t}(w_2)$. If there is and integer $r>t$ with $\deg_{x_r}(w_2)>\deg_{x_r}(w_1)$, then $w_2<_{\rm lex}x_tw_2/x_r$. As $I$ is a squarefree lexsegment ideal, we deduce that $x_tw_2/x_r\in I$. This implies that $x_tw_2^{[{\bf a}]}/x_r\in J$, as required. Therefore, suppose that for each $r>t$, we have $\deg_{x_r}(w_2)\leq\deg_{x_r}(w_1)$. This shows that $i_s<t$ and any variable in $[n]\setminus\{i_1, \ldots, i_s\}$ which divides $w_2$, also divides $w_1$. Hence,$$w_2=x_{i_1}\cdots x_{i_s}w_1/w',$$ for some monomial $w'$ dividing $w_1$. We claim that $w_1$ is divisible by $x_{i_{s-1}}$. Indeed, if $w_1$ is not divisible by $x_{i_{s-1}}$, then as $i_s<t$, we have $w_1<_{\rm lex}x_{i_s}w_1/x_t$. Therefore, $x_{i_s}w_1/x_t\in I$ is divisible by $x_{i_s}$ but not by $x_{i_{s-1}}$, contradicting our assumption. This proves our claim. Again, using our assumption, we deduce that $w_1$ is divisible by $x_{i_1}\cdots x_{i_{s-1}}$. Hence, the equality $w_2=x_{i_1}\cdots x_{i_s}w_1/w'$ can be written as $w_2=x_{i_s}w_1/w''$, where $w''=w'/(x_{i_1}\cdots x_{i_{s-1}})$. If $\deg w''=1$, then $w''=x_t$. This yields that $x_tw_2/x_{i_s}=w_1\in I$. Thus, $x_tw_2^{[{\bf a}]}/x_{i_s}\in J$, as desired. If $\deg w''\geq 2$, then $\deg w_2<\deg w_1$ and $w_2/x_{i_s}$ divides $w_1$. In other words, $w_1,w_2\in G(I)$ satisfy conditions (i), (ii) and (iii) of Lemma \ref{1necessary} with $\ell=i_s$. This is a contradiction.

Now, let $\deg_{x_t}(w_1) < \deg_{x_t}(w_2)$. Then there is a variable $x_r$ which divides $w_1$ but not $w_2$. In particular, $x_r\notin \{x_{i_1}, \ldots, x_{i_s}\}$. It follows from the choice of $t$ that $r>t$. Since $w_1<_{\rm lex} x_tw_1/x_r$ and $I$ is a squarefree lexsegment ideal, we have $x_tw_1/x_r\in I$. This yields that $x_tw_1^{[{\bf a}]}/x_r\in J$, as desired.
\hspace{6cm}
\end{proof}

We are now ready to state and prove the main result of this section.

\begin{Theorem} \label{lexsegment}
Let $I$ be a squarefree lexsegment ideal and ${\bf a} = (a_1,\ldots,a_n) \in \ZZ_{>0}^n$.  Then, the ideal $I^{[{\bf a}]}$ is componentwise linear ideal if and only if the assumptions $(\sharp)$ and $(*)$ of Lemma \ref{sufficient} are satisfied.
\end{Theorem}

\begin{proof}
Set $J:=I^{[{\bf a}]}$. According to Lemma \ref{1necessary} and \ref{sufficient}, we only need to prove that if there are integers $i< j$ with $a_i\geq 2$, $a_j\geq 2$ and a monomial $u\in G(I)$ which is divisible by $x_j$ but not by $x_i$, then $J$ is not componentwise linear. By contradiction, assume that $J$ is componentwise linear. Note that $u<_{\rm lex}x_iu/x_j$. Since $I$ is a squarefree lexsegment ideal, we have $x_iu/x_j\in I$. Set$$v_1:=u^{[{\bf a}]}\in J, \quad v_2:=(x_iu/x_j)^{[{\bf a}]}\in J.$$Let $d_1=\deg v_1$ and $d_2=\deg v_2$. The following cases arise.

\medskip

\noindent
{\bf (Case 1.)} If $d_1\geq d_2$, then set $v_2':=x_i^{d_1-d_2}v_2$ and  $k:=d_1=\deg v_2'$. Since $J$ is componentwise linear, $J_{\langle k\rangle}$ has a linear resolution. In particular, $J_{\langle k\rangle}$ is linearly related. We conclude from Lemma \ref{linrel} that the relation graph $G_{J_{\langle k\rangle}}^{(v_1,v_2')}$ is connected. Hence, there is a monomial $w\in J_{\langle k\rangle}$ in the variables dividing $x_iu$ with $\deg_{x_i}(w)=a_i-1$. It follows from the definition of $J=I^{[{\bf a}]}$ that $w/x_i^{a_i-1}\in J$. Thus, $w/x_i^{a_i-1}$ is divisible by $u^{[{\bf a}]}=v_1$. This yields that$$\deg(w)\geq (a_i-1)+d_1>k,$$where the last inequality follows from $a_i\geq 2$.  This is a contradiction.

\medskip

\noindent
{\bf (Case 2.)} If $d_1<d_2$, then set $v_1':=x_j^{d_2-d_1}v_1$ and  $k:=d_2=\deg v_1'$. As is discussed in Case 1, the relation graph $G_{J_{\langle k\rangle}}^{(v_1',v_2)}$ is connected. Hence, there is a monomial $w\in J_{\langle k\rangle}$ in the variables dividing $x_iu$ with $\deg_{x_i}(w)=a_i-1$. It follows from the definition of $J=I^{[{\bf a}]}$ that $w/x_i^{a_i-1}\in J$. Thus, $w/x_i^{a_i-1}$ is divisible by $u^{[{\bf a}]}=v_1$. This yields that$$\deg(w)\geq (a_i-1)+d_1>(a_i-a_j)+d_1=(d_2-d_1)+d_1=k,$$where the second inequality follows from $a_j\geq 2$.  This is a contradiction.
\, \, \, \,
\end{proof}

Let $2 \leq d < n$.  Recall that the {\em squarefree Veronese ideal} of $S$ of degree $d$ is the ideal $I_n(d)$ of $S$ which is generated by all squarefree monomials of $S$ of degree $d$. The following corollary is an immediate consequence of Theorem \ref{lexsegment}.

\begin{Corollary} \label{veronese}
Given ${\bf a} = (a_1,\ldots,a_n) \in \ZZ_{>0}^n$, the ideal ${I_n(d)}^{[{\bf a}]}$ is componentwise linear if and only if at most one component of ${\bf a}$ is bigger than one.
\end{Corollary}

\section{Edge ideals} \label{sec3}

Let $G$ be a finite graph having no loop, no multiple edge and no isolated vertex on the vertex set $V(G) = [n] = \{1,\ldots,n\}$ and $E(G)$ the set of edges of $G$.  Let $S=K[x_1, \ldots, x_n]$ denote the polynomial ring in $n$ variables over a field $K$.  The {\em edge ideal} of $G$ is the ideal $I(G)$ of $S$ which is generated by those monomial $x_ix_j$ with $\{i,j\}\in E(G)$.  Fr\"oberg's theorem \cite[Theorem 9.2.3]{HHgtm260} says that $I(G)$ has a linear resolution if and only if $G^c$, the complementary graph \cite[p.~153]{HHgtm260} of $G$, is a chordal graph \cite[p.~155]{HHgtm260}.  Let ${\bf a} = (a_1,\ldots,a_n) \in \ZZ_{>0}^n$ and suppose that $I^{[{\bf a}]}$ is componentwise linear.  It then follows from Corollary \ref{Cor1.3} that $G^c$ is chordal.

A {\em simplicial vertex} of $G$ is a vertex $i$ for which $$N_G(i) := \{ j \in V(G) : \{i,j\} \in E(G) \}$$ is a clique (complete subgraph) of $G$.  We say that $i \in V(G)$ is {\em adjacent} to $j \in V(G)$ if $\{i, j\} \in E(G)$.  An {\em independent set} of $G$ is a subset $\Ac \subset V(G)$ with the property that $\{i,j\} \not\in E(G)$ for all $i,j \in \Ac$ with $i\neq j$.

\begin{Lemma}
    \label{chordalonecomp}
    Let $G$ be a finite connected graph on $[n]$ and suppose that $G^c$ is chordal. If at most one component of ${\bf a}\in \ZZ_{>0}^n$ is bigger than one, then $I(G)^{[{\bf a}]}$ is componentwise linear.
\end{Lemma}

\begin{proof}
Set $I:=I(G)$.  Let ${\bf a}=(a_1, \ldots, a_n)\in \ZZ_{>0}^n$. If $a_i=1$ for each $i\in [n]$, then $I^{[{\bf a}]}=I$ has a linear resolution. So, suppose that exactly one component of ${\bf a}$, say $a_r$ is bigger than one.  We show that $I^{[{\bf a}]}$ has linear quotients (see \cite[Theorem 8.2.15]{HHgtm260}). 
Since $G^c$ is chordal, by relabeling the vertices of $G$, one may assume that each $i\in [n]$ is a simplicial vertex of $G^c\setminus\{1, \ldots, i-1\}$ \cite[Theorem 9.2.12]{HHgtm260}.  Furthermore, since a chordal graph on $n \geq 2$ vertices has at has at least two simplicial vertices, one may assume that $r=n$.  Set $N_G(n)=\{i_1, \ldots, i_k\}$, where $i_1<i_2<\cdots < i_k$. Let $u_1, \ldots, u_t$ denote those elements of $G(I)$ which are not divisible by $x_n$ and set $\{v_1, \ldots, v_k\}=G(I)\setminus\{u_1, \ldots, u_t\}$.  One may assume that $v_{\ell}=x_{i_{\ell}}x_n$ for each $\ell=1, \ldots, k$. Note that $u_i^{[{\bf a}]}=u_i$ for each $i=1, \ldots, t$. Since $(u_1, \ldots, u_t)=I(G\setminus n)$ and since $(G\setminus n)^c$ is chordal, (by relabeling $u_1, \ldots, u_t$, if necessary), one may assume that for each $j=2,\ldots, t$,  the colon ideal $$(u_1^{[{\bf a}]}, \ldots, u_{j-1}^{[{\bf a}]}):u_j^{[{\bf a}]}=(u_1, \ldots, u_{j-1}):u_j$$is generated by a subset of variables. Therefore, we only need to show that
\begin{eqnarray}
    \label{aaaaa}
(u_1, \ldots, u_t, v_1^{[{\bf a}]}, \ldots v_{\ell-1}^{[{\bf a}]}):v_{\ell}^{[{\bf a}]}
\end{eqnarray}
is generated by a subset of variables. Obviously, $(v_j^{[{\bf a}]}:v_{\ell}^{[{\bf a}]})=x_{i_j}$ for each $j=1, \ldots, \ell-1$. So, it is enough to show that for each $i=1, \ldots, t$, the monomial $\frac{u_i}{{\rm gcd}\big(u_i,v_{\ell}^{[{\bf a}]}\big)}$ is divisible by a variable belonging to (\ref{aaaaa}).
Let $u_i=x_px_q$ for distinct integers $p,q\in [n-1]$. If $x_{i_{\ell}}\in \{x_p, x_q\}$, then $\frac{u_i}{{\rm gcd}\big(u_i,v_{\ell}^{[{\bf a}]}\big)}$ is a variable.  So, suppose that $x_{i_{\ell}}\notin \{x_p, x_q\}$. We must show that at least one of $x_p,x_q$ belongs to 
(\ref{aaaaa}).  It follows from $(v_j^{[{\bf a}]}:v_{\ell}^{[{\bf a}]})=x_{i_j}$ that$$x_{i_1}, \ldots, x_{i_{\ell-1}}\in (u_1, \ldots, u_t, v_1^{[{\bf a}]}, \ldots v_{\ell-1}^{[{\bf a}]}):v_{\ell}^{[{\bf a}]}.$$

Consequently, if either $x_p$ or $x_q$ belongs to $\{x_{i_1}, \ldots, x_{i_{\ell-1}}\}$, we are done.  Suppose that neither $x_p$ nor $x_q$ belongs to $\{x_{i_1}, \ldots, x_{i_{\ell-1}}\}$. Since $G^c$ has no $4$-cycle, at least one of $x_px_n, x_qx_n,x_px_{i_{\ell}}, x_qx_{i_{\ell}}$ belongs to $G(I)$.  The following cases arise.

\medskip

\noindent
{\bf (Case 1.)} Suppose that at least one of $x_px_n$ and $x_qx_n$, say, $x_px_n$ belongs to $G(I)$. Since $x_p\in N_G(n)$, there is $m\in [k]$ with $p=i_m$. Since $x_p\notin\{x_{i_1}, \ldots, x_{i_{\ell-1}},x_{i_{\ell}}\}$,  $p=i_m>i_{\ell}$. If $x_qx_n\in G(I)$, then a similar argument shows that $q>i_{\ell}$. Thus, $p,q,i_{\ell}\in V(G\setminus\{1, 2, \cdots, i_{\ell}-1\})$. Since $i_{\ell}$ is a simplicial vertex of $G^c\setminus\{1, 2, \cdots, i_{\ell}-1\}$, we deduce that either $x_{i_{\ell}}x_p\in G(I)$ or $x_{i_{\ell}}x_q\in G(I)$. If $x_{i_{\ell}}x_p\in G(I)$, then one has
$$x_p\in (u_1, \ldots, u_t, v_1^{[{\bf a}]}, \ldots v_{\ell-1}^{[{\bf a}]}):v_{\ell}^{[{\bf a}]}.$$  If $x_{i_{\ell}}x_q\in G(I)$, then one has
$$x_q\in (u_1, \ldots, u_t, v_1^{[{\bf a}]}, \ldots v_{\ell-1}^{[{\bf a}]}):v_{\ell}^{[{\bf a}]}.$$ Therefore, suppose that $q\notin N_G(n)$. 
If $q>i_{\ell}$, the same argument as above implies the assertion. Let $q<i_{\ell}.$ Note that $q$ is a simplicial vertex of $G^c\setminus\{1, \ldots, q-1\}$. Hence, either $x_qx_n\in G(I)$ or $x_qx_{i_{\ell}}\in G(I)$. However, since $q\notin N_G(n)$, it follows that $x_qx_{i_{\ell}}\in G(I)$, which yields that$$x_q\in (u_1, \ldots, u_t, v_1^{[{\bf a}]}, \ldots v_{\ell-1}^{[{\bf a}]}):v_{\ell}^{[{\bf a}]},$$as desired. 

\medskip

\noindent
{\bf (Case 2.)} Suppose that neither $x_px_n$ nor $x_qx_n$ belongs to $G(I)$.  Therefore, at least one of $x_px_{i_{\ell}}$ and  $x_qx_{i_{\ell}}$, say, $x_px_{i_{\ell}}$ belongs to $G(I)$. It immediately follows that$$x_p\in (u_1, \ldots, u_t, v_1^{[{\bf a}]}, \ldots v_{\ell-1}^{[{\bf a}]}):v_{\ell}^{[{\bf a}]}.$$This completes the proof.
\hspace{9.2cm}
\end{proof}

\begin{Lemma}
    \label{chordaltwocomp1}
    Let $G$ be a finite connected graph on $[n]$ and suppose that $G^c$ is chordal. Given ${\bf a}=(a_1, \ldots, a_n)\in \ZZ_{>0}^n$, set $\mathcal{A}:=\{i\in [n] : a_i\geq 2\}$. If $\mathcal{A}$ is an indepent set of $G$ and if $N_G(i)\cap N_G(j)=\emptyset$ for any distinct integers $i,j\in \mathcal{A}$, then $I(G)^{[{\bf a}]}$ is componentwise linear.
\end{Lemma}

\begin{proof}
Set $I:=I(G)$. We show that $I^{[{\bf a}]}$ has linear quotients. We claim that for any pair of distinct integers $i,j\in \mathcal{A}$, any vertex in $N_G(i)$ is adjacent to any vertex in $N_G(i)$. Indeed, assume that $r\in N_G(i)$ and $s\in N_G(j)$. It follows from  $N_G(i)\cap N_G(j)=\emptyset$ that $\{j,r\}, \{i,s\}\notin E(G)$. In particular, $r\neq s$. If $\{r,s\}\notin E(G)$, then the induced subgraph of $G^c$ on $\{i,j, r, s\}$ is a $4$-cycle. This is a contradiction, as $G^c$ is a chordal graph. Therefore, our claim follows.

Let $u_1, \ldots, u_t$ denote those elements of $G(I)$ which are not divisible by any $x_i$ with $i\in \mathcal{A}$. Furthermore, for each $i\in \mathcal{A}$, Let $\mathcal{G}_i$ denote those elements of $G(I)$ which are divisible by $x_i$. Since $\mathcal{A}$ is an independent set of $G$, we conclude that $\mathcal{G}_i$'s are disjoint. Note that$$G(I)=\{u_1, \ldots, u_t\}\sqcup\bigsqcup_{i\in \mathcal{A}}\mathcal{G}_i.$$To simplify the notation, set $\mathcal{G}_i^{[{\bf a}]}=\{v^{[{\bf a}]} : v\in \mathcal{G}_i\}.$

\medskip

\noindent
{\bf (Step 1.)} For a fixed integer $i\in \mathcal{A}$, let ${\bf a_i}$ be the vectors obtained from ${\bf a}$ by deleting all components $a_k$ with $k\in \mathcal{A}\setminus\{i\}$. Set $H_i:=G\setminus (\mathcal{A}\setminus \{i\})$. Note that $I(H_i)$ is the ideal generated by $\{u_1, \ldots, u_t\}\sqcup\mathcal{G}_i$. Since $H_i^c$ is a chordal graph and exactly one component of ${\bf a_i}$ is  bigger than one, it follows from Lemma \ref{chordalonecomp} that the ideal generated by  $\{u_1^{[{\bf a}]}, \ldots, u_t^{[{\bf a}]}\}\sqcup \mathcal{G}_i^{[{\bf a}]}$ has linear quotients.

\medskip

\noindent
{\bf (Step 2.)} In view of Step 1, it is enough to show that for any two monomial $w_1\in \mathcal{G}_i$ and $w_2\in \mathcal{G}_j$ with $i\neq j$, the monomial $(w_1^{[{\bf a}]}:w_2^{[{\bf a}]})$ is divisible by a variable belonging to $(u_1^{[{\bf a}]}, \ldots, u_t^{[{\bf a}]}):w_2^{[{\bf a}]}$.  Let $w_1=x_ix_p$ and $w_2=x_jx_q$, where $p,q\in [n]$. It follows from the first paragraph of the proof that $\{p,q\}\in E(G)$. In addition, $p,q\notin \mathcal{A}$, as $\mathcal{A}$ is an independent set of $G$. Therefore, $x_px_q\in\{u_1, \ldots, u_t\}$. As a consequence,$$x_p\in(u_1, \ldots, u_t): w_2^{[{\bf a}]}=(u_1^{[{\bf a}]}, \ldots, u_t^{[{\bf a}]}):w_2^{[{\bf a}]}.$$Since $(w_1^{[{\bf a}]}:w_2^{[{\bf a}]})=x_i^{a_i}x_p$ is divisible by $x_p$, the assertion follows.
\hspace{2cm}
\end{proof}

\begin{Example}
    {\em
  Let $G$ be the finite graph on $[6]$ which is obtained by adding edges $\{1,4\}, \{2,5\}, \{3,6\}$ to the triangle on $[3]$.  Its complementary graph is chordal.
  For example,
  $$I(G)^{[(1,1,1,2,3,4)]}=(x_1x_2,x_1x_3,x_2x_3,x_1x_4^2,x_2x_5^3,x_3x_6^4)$$
  is componentwise linear.
  }
\end{Example}

\begin{Lemma}
    \label{chordaltwocomp2}
    Let $G$ be a finite connected graph on $[n]$ and suppose that $G^c$ is chordal. Assume that exactly two components $a_i$ and $a_j$ of ${\bf a}=(a_1, \ldots, a_n)\in \ZZ_{>0}^n$ are bigger than one.  Furthermore, suppose that (i) $\{i,j\}\in E(G)$, (ii) $N_G(i)\cap N_G(j)=\emptyset$, and (iii) every vertex belonging to $N_G(i)$ is adjacent to every vertex belonging to $N_G(j)$.  Then $I(G)^{[{\bf a}]}$ is componentwise linear.
\end{Lemma}

\begin{proof}
One may assume that $i=1$ and $j=2$. We show that $I(G)^{[{\bf a}]}$ has linear quotients. Let $H$ be the finite graph obtained from $G$ by deleting the edge $\{1,2\}$. We claim that $H^c$ is a chordal graph. Let $C$ be an induced cycle of $H^c$ of length at least $4$. Since $G^c$ is a chordal graph, we deduce that $\{1,2\}$ is an edge of $C$. So, one may assume that $V(C)=[p]$, where $p\leq n$ and$$E(C)=\big\{\{1,2\}, \{2,3\}, \cdots, \{p-1,p\}, \{1,p\}\big\}.$$If $p=4$, then it follows from $\{1,3\},\{2,4\}\notin E(C)$ and $\{3,4\}\in E(C)$ that $\{1,3\},\{2,4\}\in E(G)$ and $\{3,4\}\notin E(G)$, contradicting (iii). Let $p\geq 5$. We deduce from $\{1,4\},\{2,4\}\notin E(C)$ that $\{1,4\},\{2,4\}\in E(G)$, contradicting (ii). Therefore, $H^c$ is a chordal graph. Set $f=x_1x_2\in I(G)$. Then $I(G)=I(H)+(f)$. It follows from the proof of Lemma \ref{chordaltwocomp1} that $I(H)^{[{\bf a}]}$ has linear quotients. Thus, it is enough to show that $I(H)^{[{\bf a}]}:f^{[{\bf a}]}$ is generated by a subset of variables. Set $N_H(1)=\{i_1, \ldots, i_k\}$ and $N_H(2)=\{j_1, \ldots, j_m\}$. Let $u_1, \ldots, u_t$ be those monomials belonging to $G(I(H))$ which are not divisible by $x_1$ and $x_2$. Also, set $v_r:=x_1x_{i_r}$, for $r=1, \ldots, k$ and $w_s:=x_2x_{j_s}$, for $s=1, \ldots, m$. Therefore,$$I(H)=(u_1, \ldots, u_t, v_1, \ldots, v_k, w_1, \ldots, w_m).$$Note that for each $r=1, \ldots, k$, one has $(v_r^{[{\bf a}]}:f^{[{\bf a}]})=(x_1^{a_1}x_{i_r} : x_1^{a_1}x_2^{a_2})=x_{i_r}$. Thus, $x_{i_r}\in I(H)^{[{\bf a}]}:f^{[{\bf a}]}$, for each $r=1, \ldots, k$. Similarly, $x_{j_s}\in I(H)^{[{\bf a}]}:f^{[{\bf a}]}$ for each $s=1, \ldots, m$. In other words,$$(x_{i_1}, \ldots, x_{i_k}, x_{j_1}, \ldots, x_{j_m})\subseteq I(H)^{[{\bf a}]}:f^{[{\bf a}]}.$$We show that the reverse inclusion holds as well. It is enough to prove that for each $\ell=1, \ldots, t$, the monomial $u_{\ell}=(u_{\ell}:f^{[{\bf a}]})=(u_{\ell}^{[{\bf a}]}: f^{[{\bf a}]})$ is divisible by one of $x_{i_1}, \ldots, x_{i_k}, x_{j_1}, \ldots, x_{j_m}$. Let $u_{\ell}=x_qx_{q'}$, where $q,q'\in [n]\setminus \{1,2\}$ with $q \neq q'$. As $\{1,2\}, \{q,q'\}\in E(G)$ and $G^c$ has no induced $4$-cycles, at least one of $\{1,q\},\{1,q'\},\{2,q\}$ and $\{2, q'\}$ is an edge of $G$. By symmetry, assume that $\{1,q\}\in E(G)$. This yields that $q\in N_G(1)=\{i_1, \ldots, i_k\}$.  Let $q=i_h$, where $h\in [k]$. Then $u_{\ell}=x_{i_h}x_{q'}$ is divisible by $x_{i_h}$, as desired.
\hspace{5.2cm}
\end{proof}

\begin{Example}
    {\em
 Let $K_{n,m}$ denote the complete bipartite graph on the vertex set $[n+m]=[n] \sqcup ([n+m]\setminus [n])$.  Let ${\bf a}=(a_1, \ldots, a_{n+m})\in \ZZ_{>0}^{n+m}$ and suppose that at most one of $a_1, \ldots, a_n$ is bigger than $1$ and that at most one of $a_{n+1}, \ldots, a_{n+m}$ is bigger than $1$.  Then $I(G)^{[{\bf a}]}$ is componentwise linear.
    }
\end{Example}

\begin{Lemma}
    \label{chordaltwocompnot1}
    Let $G$ be a finite connected graph on $[n]$ and suppose that $G^c$ is chordal. Let ${\bf a}=(a_1, \ldots, a_n)\in \ZZ_{>0}^n$. If there are $i,j\in [n]$ with $i \neq j$ for which $a_i \geq 2, a_j\geq 2$ and $N_G(i)\cap N_G(j)\neq \emptyset$, then  $I(G)^{[{\bf a}]}$ is not componentwise linear.
\end{Lemma}

\begin{proof}
Set $J:=I(G)^{[{\bf a}]}$. Suppose on the contrary that $J$ is componentwise linear.  One may assume that $a_i\geq a_j$. Let $p\in N_G(i)\cap N_G(j)$.  Set $u:=x_i^{a_i}x_p^{a_p}$ and $v:=x_j^{a_i}x_p^{a_p}$. Since $a_i\geq a_j$, one has $u,v\in J$. Let $k:=\deg(u)=\deg(v)$. As $J$ is componentwise linear, the ideal $J_{\langle k\rangle}$ has linear resolution. In particular, it is linearly related. It follows from Lemma \ref{linrel} that the relation graph $G_{J_{\langle k\rangle}}^{(u,v)}$ is connected. Therefore, there is a monomial $w\in J_{\langle k\rangle}$ in the variables $x_i, x_j, x_p$ for which $\deg_{x_i}(w)=a_i-1$. Thus, $w$ is divisible by $x_j^{a_j}x_p^{a_p}$.  It then follows that$$\deg(w)\geq (a_i-1)+a_j+a_p>k,$$where the last inequality follows from $a_j\geq 2$. This is a contradiction.
\end{proof}

\begin{Lemma}
    \label{chordaltwocompnot2}
    Let $G$ be a finite connected graph on $[n]$ and suppose that $G^c$ is chordal. Let ${\bf a}=(a_1, \ldots, a_n)\in \ZZ_{>0}^n$. Assume that there are $i,j\in [n]$ with $i \neq j$  for which $\{i,j\}\in E(G)$ and $a_i \geq 2,a_j\geq 2$. If there are $p\in N_G(i)$ and $q\in N_G(j)$ with $\{p,q\}\notin E(G)$, then $I(G)^{[{\bf a}]}$ is not componentwise linear.
\end{Lemma}

\begin{proof}
If $N_G(i)\cap N_G(i)\neq\emptyset$, then the assertion follows from Lemma \ref{chordaltwocompnot1}. Set $J:=I(G)^{[{\bf a}]}$. Let $N_G(i)\cap N_G(i)=\emptyset$. In particular, $p\neq q$ and $\{i,q\}, \{j.p\}\notin E(G)$. Suppose on the contrary that $J$ is componentwise linear. Set $u:=x_i^{a_i}x_p^{a_p}$ and $v:=x_j^{a_j}x_q^{a_q}$. Then $u,v\in J$. By symmetry, assume that $\deg(u)\geq \deg(v)$. Define $v':=x_j^{\deg(u)-\deg(v)}v$ and $k:=\deg(u)=\deg(v')$. By the same argument as in the proof of Lemma \ref{chordaltwocompnot1}, the relation graph $G_{J_{\langle k\rangle}}^{(u,v')}$ is connected. Therefore, there is a monomial $w\in J_{\langle k\rangle}$ in the variables $x_i, x_j, x_p, x_q$ for which $\deg_{x_j}(w)=a_j-1$. Since $\{p,q\}, \{i,q\}\notin E(G)$, it follows that $w$ is divisible by $x_i^{a_i}x_p^{a_p}$. Since $a_j\geq 2$, it then follows that $$\deg(w)\geq (a_j-1)+a_i+a_p>k,$$
a contradiction.
\hspace{10.7cm}
\end{proof}

\begin{Lemma}
    \label{chordalthreecompnot}
    Let $G$ be a finite connected graph on $[n]$ and suppose that $G^c$ is chordal. Let ${\bf a}=(a_1, \ldots, a_n)\in \ZZ_{>0}^n$. Assume that there are distinct integers $i,j,k\in [n]$ with $\{i,j\}\in E(G)$ for which $a_i\geq 2, a_j\geq 2$ and $a_k\geq 2$.  Then $I(G)^{[{\bf a}]}$ is not componentwise linear.
\end{Lemma}

\begin{proof}
Set $J:=I(G)^{[{\bf a}]}$.  If $\{i,k\}\in E(G)$ or $\{j,k\}\in E(G)$, then the assertion follows from Lemma \ref{chordaltwocompnot1}. Let $\{i,k\}, \{j,k\}\notin E(G)$. As $k$ is not an isolated vertex of $G$, there is $p\in [n]\setminus\{i,j,k\}$ for which $\{p,k\}\in E(G)$. If either $\{i,p\}\in E(G)$ or $\{j,p\}\in E(G)$, then the assertion, again, follows from Lemma \ref{chordaltwocompnot1}. Let $\{i,p\}, \{j,p\}\notin E(G)$. Set $u:=x_i^{a_i}x_j^{a_j}, v:=x_p^{a_p}x_k^{a_k}, d_u:=\deg(u)$ and ${d_v}:=\deg(v)$.  Suppose on the contrary that $J$ is componentwise linear.  The following two cases arise.

\medskip

\noindent
{\bf (Case 1.)} Let $d_u\geq d_v$.  Set $v':=x_k^{d_u-d_v}v$ and $k:=d_u=\deg(v')$. By the same argument as in the proof of Lemma \ref{chordaltwocompnot1}, the relation graph $G_{J_{\langle k\rangle}}^{(u,v')}$ is connected. Therefore, there is a monomial $w\in J_{\langle k\rangle}$ in the variables $x_i, x_j, x_p, x_k$ for which $\deg_{x_k}(w)=a_k-1$. Since $\{i,p\}, \{j,p\}\notin E(G)$, it follows that $w$ is divisible by $u=x_i^{a_i}x_j^{a_j}$. Since $a_k\geq 2$, it then follows that $$\deg(w)\geq (a_k-1)+d_u>k,$$a contradiction.

\medskip

\noindent
{\bf (Case 2.)} Let $d_u< d_v$.  Set $u':=x_i^{d_v-d_u}u$ and $k:=d_v=\deg(u')$. Considering the relation graph $G_{J_{\langle k\rangle}}^{(u',v)}$, a similar argument as in Case 1 yields a contradiction.
\end{proof}

Finally, summarizing Lemmas \ref{chordalonecomp}, \ref{chordaltwocomp1},  \ref{chordaltwocomp2}, \ref{chordaltwocompnot1}, \ref{chordaltwocompnot2} and \ref{chordalthreecompnot} yields the following combinatorial characterization for $I(G)^{[{\bf a}]}$ to be componentwise linear.

\begin{Theorem}
    \label{chordal}
    Let $G$ be a finite connected graph on $[n]$ and suppose that $G^c$ is chordal. Given ${\bf a} \in \ZZ_{>0}^n$, the ideal $I(G)^{[{\bf a}]}$ is componentwise linear if and only if ${\bf a} \in \ZZ_{>0}^n$ satisfies the conditions given in one of Lemmata \ref{chordalonecomp}, \ref{chordaltwocomp1} and \ref{chordaltwocomp2}.
\end{Theorem}

\section*{Statements and Declarations}
The authors have no Conflict of interest to declare that are relevant to the content of this article.

\section*{Data availability}
Data sharing does not apply to this article as no new data were created or analyzed in this study.

\end{document}